\documentclass{article}
\usepackage[utf8]{inputenc}
\usepackage{subfig}
\usepackage{graphicx} % Required for inserting images
\usepackage{caption,color}   % 24.62
\usepackage{amssymb}
\usepackage{amsthm}
\usepackage{amsmath}
\usepackage{epic}
\usepackage{setspace}
\usepackage{float}
\usepackage{natbib}
\usepackage{multirow}
\usepackage{hyperref}
\usepackage{xcolor}
\usepackage{marginnote}
\usepackage{amsthm}
\usepackage{authblk}

\newtheorem{thm}{Theorem}[section]
\newtheorem{cor}[thm]{Corollary}

\newtheorem{lem}[thm]{Lemma}
\newtheorem{pro}[thm]{Proposition}
\newtheorem{defi}[thm]{Definition}

\title{\textbf{The characteristic polynomial of sunflowers}}

\author{Changjiang Bu\thanks{Corresponding author. buchangjiang@hrbeu.edu.cn}$^{1}$,  Lixiang Chen\thanks{clxmath@163.com}$^{2}$,  Ge Lin\thanks{linge0717@126.com}}

\affil[1]{School of Mathematical Sciences, Harbin Engineering University, Harbin, PR China }
\affil[2]{Center for Combinatorics, LPMC, Nankai University, Tianjin, PR China}

\date{}

\begin{document}

\maketitle

\begin{abstract}
%Let $S=S(k,s,p)$  be the $k$-uniform sunflower hypergraph with $s$ seeds and $p$ petals.
A uniform hypergraph is called a sunflower if all of its hyperedges intersect in the same set of vertices.
In this paper, we determine the eigenvalues and spectral moments of a sunflower, thereby obtaining an explicit formula for its characteristic polynomial.
\end{abstract}


\noindent\emph{AMS classification(2020):} 05C50, 05C65

\section{Introduction}

In spectral hypergraph theory, determining the characteristic polynomial of a hypergraph is a fundamental problem,
which is equivalent to determining all eigenvalues and their multiplicities.
The characteristic polynomial of a uniform hypergraph is the resultant of a system of homogeneous multi-linear equations \cite{qi2005eigenvalues}.
%Determining the characteristic polynomial of a hypergraph, defined as the resultant of a system of homogeneous multi-linear equations \cite{qi2005eigenvalues}, is a fundamental problem in spectral hypergraph theory.
The resultant both determines whether such a system has non-trivial solutions and provides a way to express those solutions explicitly \cite{cox2005using},
which is studied and applied in various fields like algebraic geometry and number theory.
%The resultant is studied and applied in various fields like algebraic geometry and number theory, because it both determines whether polynomial equations have non-trivial solutions and provides a way to express those solutions explicitly \cite{cox2005using}.
However, computing the resultant (and thus, the characteristic polynomial of a hypergraph) is NP-hard in general \cite{hillar2013most}.

To the best of our knowledge, there are currently three tools for computing the characteristic polynomial of a hypergraph.
The first one is a resultant formula given in the elimination theory \cite{macaulay1902some},
which expresses the resultant as a quotient of two determinants.
Li, Su, and Fallat express the characteristic polynomials of a uniform hypertree in terms of the matching polynomials of its subhypertrees in this way \cite{li2024hypertree}.
The second one is the Poisson formula, which may provide an inductive computational method \cite{jouanolou1991formalisme}.
It is applied to derive the characteristic polynomials of some concrete hypergraphs, including hyperstars \cite{bao2020combinatorial}, hyperpaths \cite{chen2021reduction}, and $3$-uniform complete hypergraphs \cite{zheng2021complete}.
The third tool is the so-called generalized trace or spectral moment,
which is defined by Morozov and Shakirov \cite{morozov2011analogue}.
Clark and Cooper establish a hypergraph analogue of the Harary–Sachs Theorem \cite{clark2021harary}.
Shao, Qi, and Hu \cite{shao2015some} obtain a useful expression for the trace of general tensors.
Using the formula of Shao et al.,
Chen, van Dam, and Bu derive an explicit expression for the characteristic polynomial of power hypergraphs \cite{chen2024spectra}.

A hypergraph is called linear if every pair of hyperedges intersects in at most one vertex.
So far,
most hypergraphs whose characteristic polynomials are explicitly known are linear.
A uniform hypergraph is called a \emph{sunflower} if all of its hyperedges intersect in the same set of vertices, called its seeds.
Let $\mathcal{S}(k,s,p)$ denote a $k$-uniform sunflower with $s$ seeds and $p$ petals.
When $s \geq 2$, the sunflowers $\mathcal{S}(k,s,p)$ form a family of non-linear hypergraphs.

The characteristic polynomial of a sunflower is first studied by Cooper and Dutle,
who provide the characteristic polynomial of a $3$-uniform hyperstar $\mathcal{S}(3,1,p)$ \cite{cooper2015computing}.
Bao, Fan, Wang, and Zhu later generalize the result to $\mathcal{S}(k,1,p)$ for general $k$ \cite{bao2020combinatorial}.
In this paper, we further extend the result to the more general case $\mathcal{S}(k,s,p)$ (see Theorem \ref{mainresult}),
by determining its eigenvalues (see Theorem \ref{dingli1}) and spectral moments (see Theorem \ref{dingli2}).

\section{Preliminaries}

In this section,
we introduce some basic notation and auxiliary lemmas on the eigenvalues and spectral moments of hypergraphs.

A hypergraph $H=(V,E)$ is called $k$-\emph{uniform} if each edge of $H$ contains exactly $k$ vertices.
For a $k$-uniform hypergraph $H$ with $V=[n]:=\{1,\ldots,n\}$,
its \emph{adjacency tensor} $A_{H}=(a_{i_{1}i_{2}\cdots i_{k}})$ is a $k$-order $n$-dimensional tensor \cite{cooper2012spectra}, where
\begin{equation*}
a_{i_{1}i_{2}\cdots i_{k}}=\begin{cases}
\frac{1}{(k-1)!}&\text{if $\{i_{1},i_{2},\ldots,i_{k}\}\in E$},\\
0&\text{otherwise}.
\end{cases}
\end{equation*}
For a vector $\mathbf{x}=(x_{1},\ldots,x_{n})^{\top}\in\mathbb{C}^{n}$ and $S \subseteq V$,
let $\mathbf{x}^{S}=\prod_{s \in S}x_{s}$.
In particular,
we set $\mathbf{x}^{\emptyset}=1$.
Let $E_{H}(v)$ denote the set of hyperedges of $H$ containing the vertex $v$.
For some $\lambda \in \mathbb{C}$,
if there exists a non-zero vector $\mathbf{x}$ such that for each $i \in [n]$,
\begin{align*}
\lambda x_{i}^{k-1}=\sum_{i_{2},\ldots,i_{k}=1}^{n}a_{i i_{2} \cdots i_{k}}x_{i_{2}}\cdots x_{i_{k}},
\end{align*}
or equivalently, for each $v\in V$,
\begin{align}\label{eq2}
\lambda x_{v}^{k-1}=\sum_{e \in E_{H}(v)}\mathbf{x}^{e\setminus\{v\}},
\end{align}
then $\lambda$ is called an \emph{eigenvalue} of $H$ and $\mathbf{x}$ is an \emph{eigenvector} of $H$ corresponding to $\lambda$ (\cite{lim2005singular,qi2005eigenvalues}).
The \emph{characteristic polynomial} $\phi_{H}(\lambda)$ of $H$ is defined as the resultant of the polynomial system
$\{ \lambda x_{v}^{k-1}-\sum_{e \in E_{H}(v)}\mathbf{x}^{e\setminus\{v\}} : v \in V \}$ \cite{qi2005eigenvalues}.
%Consequently, $\lambda$ is an eigenvalue of $H$ if and only if $\phi_{H}(\lambda)=0$.

The sum of $d$-th powers of all eigenvalues of a uniform hypergraph $H$ is called the $d$-th order spectral moment of $H$,
denoted by $\mathrm{S}_{d}(H)$.
It is known that the $d$-th order spectral moment of a graph is equal to the $d$-th order trace of the adjacency matrix \cite{cvetkovic1980spectra}.
Similar to the case of graphs,
the $d$-th order spectral moment $\mathrm{S}_{d}(H)$ is equal to the $d$-th order trace of the adjacency tensor $A_{H}$ \cite{hu2013determinants}.

In order to describe the formula for the spectral moments of a hypergraph (see Lemma \ref{yinli1}),
we introduce some related notation.
%Let $H=(V,E)$ be a $k$-uniform hypergraph with $V=[n]$.
%For a $k$-tuple $i v_1v_2\cdots v_{k-1}\in [n]^k$,
For a $k$-uniform hypergraph $H$ with $V=[n]$ and a $k$-tuple $i v_1v_2\cdots v_{k-1}\in [n]^k$,
we interpret it as a rooted hyperedge of $H$ with root $i$ if $\{i, v_1,v_2,\ldots,v_{k-1}\} \in E$.
When no confusion arises,
we use $i v_1v_2\cdots v_{k-1}$ to represent the hyperedge $\{i, v_1,v_2,\ldots,v_{k-1}\}$ rooted at $i$.
Let $\mathcal{F}_{d}=\{ (i_{1}\alpha_{1},\ldots,i_{d}\alpha_{d}): 1 \leq i_{1} \leq \cdots\leq i_{d} \leq n, \alpha_{1},\ldots,\alpha_{d}\in[n]^{k-1}, i_j\alpha_j \in E \}$.
Let $f=(i_{1}\alpha_{1},\ldots,i_{d}\alpha_{d})\in\mathcal{F}_{d}$,
where $i_{j}\alpha_{j}\in[n]^{k}$ and $j=1,\ldots,d$.
It implies that $f$ consists of $d$ rooted hyperedges.
Construct a $k$-uniform hypergraph $H_{f}$ such that $V(H_{f})=\bigcup_{j=1}^{d}i_j\alpha_j$ and $E(H_{f})=\bigcup_{j=1}^{d}\{i_j\alpha_j\}$.
Obviously, $H_{f}$ is a subhypergraph of $H$.
For $i_{j}\alpha_{j}=i_{j}v_{1}\cdots v_{k-1}$,
let $\theta(i_{j}\alpha_{j})=\{(i_{j},v_{1}),\ldots,(i_{j},v_{k-1})\}$ be the set of arcs from $i_{j}$ to $v_{1},\ldots,v_{k-1}$, where $(v_{1},v_{2})$ denotes an arc from vertex $v_{1}$ to vertex $v_{2}$.
Construct a multi-digraph $D_{f}$ such that $V(D_{f})=\bigcup_{j=1}^{d}i_j\alpha_j$ and $E(D_f)=\bigcup_{j=1}^{d}\theta(i_{j}\alpha_{j})$.

Consider next a set of representatives of isomorphic such connected subhypergraphs
\begin{align*}
\mathcal{H}_{d}=\{ \widehat{H}: \mbox{$H_{f}\cong\widehat{H}$ and $D_f$ is Eulerian for some $f\in\mathcal{F}_{d}$} \}.
\end{align*}
For $\widehat{H}\in\mathcal{H}_{d}$,
let $\mathcal{F}_{d}(\widehat{H})=\{f \in \mathcal{F}_{d} :H_{f}\cong\widehat{H}\}$.
Denote the number of subhypergraphs of $H$ which are isomorphic to $\widehat{H}$ by $N_{H}(\widehat{H})$.
Let the set of representatives of isomorphic Eulerian multi-digraphs
\[\mathfrak{D}_{d}(\widehat{H})=\{D:\mbox{$D_{f}\cong D$ is  Eulerian for some $f\in\mathcal{F}_{d}(\widehat{H})$}\}.\]

\begin{lem}\cite{chen2024spectra}\label{yinli1}
The $d$-th order spectral moment of a $k$-uniform hypergraph $H$ is
\begin{align*}
\mathrm{S}_{d}(H)=(k-1)^{n-1}\sum_{\widehat{H}\in\mathcal{H}_{d}}c_{d}(\widehat{H})N_{H}(\widehat{H}).
\end{align*}
The $d$-th order spectral moment coefficient of the subhypergraph $\widehat{H}$ is
\begin{align*}
c_{d}(\widehat{H})=d(k-1)((k-1)!)^{-d}\sum_{D\in\mathfrak{D}_{d}(\widehat{H})}\frac{|\{f:\mbox{$f\in\mathcal{F}_{d}(\widehat{H})$ and $D_{f}\cong D$}\}|t(D)}{\prod_{v\in V(D)}\mathrm{deg}^{+}(v)},
\end{align*}
where $\mathrm{deg}^{+}(v)$ is the out-degree of $v \in V(D)$ and $t(D)$ is the number of spanning trees of $D$.
\end{lem}

\section{Spectra of a sunflower}

\subsection{The eigenvalues of a sunflower}

For a sunflower $\mathcal{S}=\mathcal{S}(k,s,p)$,
denote by $S$ the set of seed vertices, and let $P_{i}$ be the set of vertices on the $i$-th petal of $\mathcal{S}$ for each $i\in[p]$.
Let $(\lambda,\mathbf{x})$ be an eigenpair of $\mathcal{S}$.
From \eqref{eq2}, for each $v \in S$, we have that
\begin{align}\label{shi1}
\lambda x_{v}^{k-1}=\sum_{i=1}^{p}\mathbf{x}^{S\setminus\{v\}}\mathbf{x}^{P_{i}},
\end{align}
and for each $i\in[p]$ and each $u\in P_{i}$, we have that
\begin{align}\label{shi2}
\lambda x_{u}^{k-1}=\mathbf{x}^{S}\mathbf{x}^{P_{i}\setminus\{u\}}.
\end{align}
For $k=2$,
it is known that the star graph $\mathcal{S}(2,1,p)$ have the spectrum
$$\{[\sqrt{p}]^{1}, [-\sqrt{p}]^{1}, [0]^{p-1}\}$$ (exponents indicate multiplicity) \cite{cvetkovic1980spectra}.
For $k\geq3$,
we first determine all the eigenvalues of $\mathcal{S}(k,s,p)$ as follows.

\begin{thm}\label{dingli1}

For $k\geq3$,
let $\mathcal{S}=\mathcal{S}(k,s,p)$ be a $k$-uniform sunflower with $s$ seeds and $p$ petals.
Denote $\Xi^{p}=\{(\xi_{i})\in\mathbb{C}^{p}:\mbox{$\xi_{i}^{s+1}-\xi_{i}=0$ for all $i\in[p]$} \}$ and $\Xi^{p}_{0}=\{\xi\in \Xi^{p}: \mbox{$|\mathrm{supp}(\xi)|=0$ or $p$} \}$.
Let $\mathbf{e}_{p}$ denote the $p$-dimensional all-ones column vector.
Then the complex number $\lambda$ is an eigenvalue of $\mathcal{S}$ if and only if
\begin{enumerate}
\renewcommand{\labelenumi}{(\alph{enumi})}

\item  $\prod_{\xi\in\Xi^{p}}(\lambda^{k}-(\mathbf{e}_{p}^{\top}\xi)^{s})=0$, when $s<k-1$;

\item $\prod_{\xi\in\Xi^{p}_{0}}(\lambda^{k}-(\mathbf{e}_{p}^{\top}\xi)^{k-1})=0$, when $s=k-1$.

\end{enumerate}
\end{thm}

\begin{proof}
First of all,
it is known that a $k$-uniform hypergraph always has an eigenvalue $0$ for $k\geq3$ \cite{qi2014h}.
On the other hand,
$\lambda=0$ is a solution of both $\prod_{\xi\in\Xi^{p}}(\lambda^{k}-(\mathbf{e}_{p}^{\top}\xi)^{s})=0$ and $\prod_{\xi\in\Xi_{0}^{p}}(\lambda^{k}-(\mathbf{e}_{p}^{\top}\xi)^{k-1})=0$.
Thus, for the remainder of the proof, we only need to consider the case $\lambda\neq0$.

Let $(\lambda,\mathbf{x})$ be an eigenpair of $\mathcal{S}$ with $\lambda\neq0$.
Using \eqref{shi2}, for any $i\in [p]$, we have that
\begin{align*}
    \lambda^{k-s}(\mathbf{x}^{P_{i}})^{k-1}&=\prod_{u \in P_i}\lambda x_{u}^{k-1}\\
    &=\prod_{u \in P_i}\mathbf{x}^{S}\mathbf{x}^{P_{i}\setminus\{u\}}\\
    &=(\mathbf{x}^{S})^{k-s}(\mathbf{x}^{P_{i}})^{k-s-1},
\end{align*}
that is, $(\mathbf{x}^{P_{i}})^{k-s-1}\left(\lambda^{k-s}(\mathbf{x}^{P_{i}})^{s}-(\mathbf{x}^{S})^{k-s}\right)=0$.
We observe that $\mathbf{x}^{S} \neq 0$;
otherwise, from \eqref{shi2} and $\lambda \neq 0$, it would follow that $\mathbf{x} = \mathbf{0}$,
which is not the case.
Hence, we have that
\begin{align*}
(\mathbf{x}^{S}\mathbf{x}^{P_{i}})^{k-1-s}\left(\lambda^{k-s}(\mathbf{x}^{S}\mathbf{x}^{P_{i}})^{s}-(\mathbf{x}^{S})^{k}\right)=0. %\mbox{($\lambda(\mathbf{x}_{V_{0}}\mathbf{x}_{V_{i}})^{k-1}-\mathbf{x}_{V_{0}}^{k}=0$, if $s=k-1$)}.
\end{align*}
Let $\beta \in \mathbb{C}$ be such that $\lambda^{k-s}\beta^{s}-(\mathbf{x}^{S})^{k}=0$. Then we have that $\mathbf{x}^{S}\mathbf{x}^{P_{i}}=\beta \xi_i$ for some $\xi_i$ satisfying
\begin{align}\label{eq1}
    \begin{cases}
        \xi_{i}^{s+1}-\xi_{i}=0, & k-1-s >0,\\
        \xi_{i}^{s}-1=0, &k-1-s=0.
    \end{cases}
\end{align}
So, using \eqref{shi1}, we get $\lambda x_v^k=\sum_{i=1}^{p}\mathbf{x}^{S}\mathbf{x}^{P_{i}}$ for every $v\in S$.
Using $\mathbf{x}^{S}\mathbf{x}^{P_{i}}=\beta \xi_i$ and $\beta^{s}=\frac{(\mathbf{x}^{S})^{k}}{\lambda^{k-s}}$,
we have
\begin{align*}
\lambda^s(\mathbf{x}^{S})^k&=\prod_{v \in S}\lambda x_v^k=\prod_{v \in S}\sum_{i=1}^{p}\mathbf{x}^{S}\mathbf{x}^{P_{i}}=\beta^s( \mathbf{e}_{p}^{\top}\xi)^s=\frac{(\mathbf{x}^{S})^k(\mathbf{e}_{p}^{\top}\xi)^s}{\lambda^{k-s}},
\end{align*}
that is,
\begin{align*}
\lambda^{k}-(\mathbf{e}_{p}^{\top}\xi)^{s}=0,
\end{align*}
for some $\xi =(\xi_i) \in\mathbb{C}^{p}$ satisfying \eqref{eq1}.

Next, we prove the following implication: If $\lambda \in \mathbb{C}$ satisfies $\lambda^{k}-(\mathbf{e}_{p}^{\top}\xi)^{s}=0$ for some $\xi = (\xi_i) \in\Xi^{p}$ ($\xi\in\Xi_{0}^{p}$, if $s=k-1$), then $\lambda$ is an eigenvalue of $\mathcal{S}$.
We may assume that $\mathbf{e}_{p}^{\top}\xi \neq 0$, since $0$ is trivially an eigenvalue of $\mathcal{S}$.
Let $\mu \in\mathbb{C}$ be such that $\mu^{k}=\mathbf{e}_{p}^{\top}\xi =\sum_{i=1}^p \xi_i$, and let $\gamma_i \in\mathbb{C}$ be such that $\gamma_i^k =\xi_i$ for every $i \in [p]$.
We fix $u_{i}$ as one of the vertices in $P_{i}$. Note that $P_{i}\setminus\{u_{i}\}$ is empty if $s=k-1$. Using these, we can construct an eigenvector corresponding to $\lambda$ as follows. Let $\mathbf{x}$ be a vector with entries
\begin{equation*}
x_{v}=\begin{cases}
1, &\text{for $v\in S$},\\
\frac{\lambda \gamma_i^{s+1}}{\mu^{s+1}}, &\text{for $v=u_{i}\in P_{i}$ and $i\in[p]$},\\
\frac{\gamma_i}{\mu}, &\text{for $v\in P_{i}\setminus\{u_{i}\}$ and $i\in[p]$}.
\end{cases}
\end{equation*}
We now verify that the pair $(\lambda,\mathbf{x})$ satisfies \eqref{shi1} and \eqref{shi2}, which will complete the proof.

For each $v\in S$, we have that
\begin{align*}
\lambda x_{v}^{k-1}=\sum_{i=1}^{p}\frac{\lambda \xi_i}{\sum_{i=1}^p \xi_i}=
\sum_{i=1}^{p}\frac{\lambda \gamma_i^k}{\mu^k}=\sum_{i=1}^{p}\frac{\lambda \gamma_i^{s+1}}{\mu^{s+1}}\left(\frac{\gamma_i}{\mu}\right)^{k-s-1}=\sum_{i=1}^p\mathbf{x}^{S\setminus\{v\}}\mathbf{x}^{P_{i}},
\end{align*}
which shows  \eqref{shi1}.

For each $i\in[p]$, using $\gamma_i^{k(s+1)}=\gamma_i^k$ (whenever $\xi_i^{s+1}=\xi_i$ or $\xi_i^{s}=1$) and $\mu^{sk}=\lambda^k$, we have that
\begin{align*}
\lambda x_{u_{i}}^{k-1}=\lambda\left(\frac{\lambda\gamma_i^{s+1}}{\mu^{s+1}}\right)^{k-1}= \left(\frac{\gamma_i}{\mu}\right)^{k-s-1}=\mathbf{x}^{S}\mathbf{x}^{P_{i}\setminus\{u_{i}\}},
\end{align*}
and for each $u\in P_{i}\setminus\{u_{i}\}$ (if $s < k-1$),
we have that
\begin{align*}
\lambda x_{u}^{k-1}
=\lambda\left(\frac{\gamma_i}{\mu}\right)^{k-1}=\frac{\lambda\gamma_i^{s+1}}{\mu^{s+1}}\left(\frac{\gamma_i}{\mu}\right)^{k-s-2}=\mathbf{x}^{S}\mathbf{x}^{P_{i}\setminus\{u\}}.
\end{align*} These two equalities together verify \eqref{shi2}.
\end{proof}

\subsection{The spectral moments of a sunflower}

The spectrum of a hypergraph is said to be $k$-symmetric if it is invariant under a rotation of an angle $\frac{2\pi}{k}$ in the complex plane \cite{cooper2012spectra}.
Note that every hyperedge of a sunflower $\mathcal{S}=\mathcal{S}(k,s,p)$ contains a vertex of degree one,
then $\mathcal{S}$ is the so-called cored hypergraph \cite{hu2013cored}.
It is known that the spectrum of a $k$-uniform cored hypergraph is $k$-symmetric \cite{shao2015some},
hence the $d$-th order spectral moment $\mathrm{S}_{d}(\mathcal{S})=0$ for $k\nmid d$.
We restrict our attention to the case $k \mid d$.

We will use Lemma \ref{yinli1}, the formula for the spectral moment for general hypergraphs,
to give an expression for $\mathrm{S}_{d}(\mathcal{S})$,
and consequently the characteristic polynomial of $\mathcal{S}$.
The key step is to give an explicit characterization of the multi-digraph $D_f$ involved in the formula,  which determines the family of sub-hypergraphs $\mathcal{H}_d$ and the spectral moment coefficient $c_{d}(\widehat{H})$ for $\widehat{H} \in \mathcal{H}_d$.

For two vertex sets $V$ and $U$, we write $V \xrightarrow{m} U$ to denote that there is an arc of multiplicity $m$ from each vertex $v \in V$ to every vertex in $U \setminus\{v\}$. Here, $m$ is a nonnegative integer; in particular, we define $m=0$ to indicate the absence of such arcs.

\begin{lem}\label{yinli3.2}
For a sunflower $\mathcal{S}=\mathcal{S}(k,s,p)$,
let $S$ be the set of seed vertices, and let $P_{i}$ be the set of vertices on the $i$-th petal of $\mathcal{S}$ for all $i\in[p]$.
Let $d$ be a positive integer such that $k \mid d$ and $f \in \mathcal{F}_d(\mathcal{S})$.
%Let $f \in \mathcal{F}_d$.
Then

\begin{enumerate}
\renewcommand{\labelenumi}{(\alph{enumi})}
\item Let $\mathrm{deg}^+(v)$ and $\mathrm{deg}^-(v)$ be the out-degree and in-degree of the vertex $v$ in the multi-digraph $D_f$. Then $\mathrm{deg}^+(v) =\mathrm{deg}^-(v)$  for all $v \in V(D_f)$.

\item  For each $v \in S$ and each $i \in [p]$,  there exists a non-negative integer $q_{vi}$ such that $\{v\} \xrightarrow{q_{vi}}  P_{i}$ in $D_f$.

\item For each $i \in [p]$,  there exists a positive integer $m_i$ such that $P_{i} \xrightarrow{m_i} S \cup P_{i}$ in $D_f$.

\item When $s \geq 2$, then $S \xrightarrow{\frac{d}{k}} S$ in $D_f$ .
%\item There exists a positive integer $m$ such that $S \xrightarrow{m} S$ in $D_f$ if $s \geq 2$.

\end{enumerate}
\end{lem}

\begin{proof}

Note that the multi-digraph $D_f$ is Eulerian, it implies (a).

For each $i \in [p]$, let $e_i =S \cup P_i \in E(\mathcal{S})$.
For each $v \in S$, denote by $q_{vi}$ the number of $k$-tuples $v\alpha$ in $f$ such that the first entry is $v$ and $v\alpha=e_i$.
From the construction of $D_{f}$,  we have (b).

For each $i \in [p]$ and $u \in P_{i}$,
note that the $k$-tuples $u\alpha$ in $f$ with initial entry $u$ are exactly the hyperedge $e_i$ rooted at $u$, that is, $u\alpha=e_i$.
Let $m(v)$ denote the number of such $k$-tuples in $f$ with $v$ as their first entry for any $v \in V(D_{f})$.
For any $u \in P_{i}$,  we have $\mathrm{deg}^+(u)=(k-1)m(u)$
and $\mathrm{deg}^-(u)=\sum_{v\in P_{i}\setminus\{u\}}m(v)+\sum_{v\in S}q_{vi}$
(or $\mathrm{deg}^-(u)=\sum_{v\in S}q_{vi}$ if $s=k-1$).
Since $\mathrm{deg}^+(u)>0$, it follows that $m(u) >0$.
By (a),
we get
\begin{align}\label{eqpi}
km(u)=\sum_{v\in P_{i}}m(v)+\sum_{v\in S}q_{vi},
\end{align}
where the right-hand side is independent of $u$. It implies that $m(u)$ is constant for all $u \in P_{i}$.
For each $i \in [p]$ and all $u \in P_{i}$, we set $m(u)=m_{i}$.
From the construction of $D_{f}$, we have (c).

For each $v \in S$, observe that $v$ appears in every rooted hyperedge in $f$.
Note that $m(v)$ counts the number of such hyperedges rooted at $v$,
and thus $d-m(v)$ counts the number of such hyperedges not rooted at $v$.
Then we have $\mathrm{deg}^+(v)=(k-1)m(v)$ and $\mathrm{deg}^-(v)=d-m(v)$.
%Because $v$ is contained in all $k$-tuples in $f$ that do not have it as the first entry,
%it follows that $\mathrm{deg}^-(v)=d-m_{S}(v)$.
By (a),
we get $m(v)=\frac{d}{k}$ for each $v \in S$.
When $s \geq 2$, for any two distinct vertices $v_1,v_2 \in S$, the multiplicity of arc in $D_f$ from $v_1$ to $v_2$ is $m(v_1)=\frac{d}{k}$, it implies (d).
\end{proof}

For positive integers $m$ and $t$, let $mK_t$ denote the complete multi-digraph on $t$ vertices, where each arc has multiplicity exactly $m$.
In particular, for any $m$, $mK_1$ is defined as a single vertex.

\begin{defi}\label{defi3.3}
Let $\mathbf{m}=(m_i)$ be a $p$-dimensional positive integer column vector and $Q=(q_{vi})$ be an $s \times p$ non-negative integer matrix.
Let $d$ be a positive integer such that $k \mid d$.
The multi-digraph $D(\mathbf{m},Q)$ is constructed as follows:
\begin{itemize}
\item The vertex set of $D(\mathbf{m},Q)$ can be partitioned into disjoint $p+1$ subsets: $V_0,V_1,\ldots,V_p$, where $|V_0|=s$ and $|V_i|=k-s$ for all $i \in [p]$.
\item The induced subgraph $D[V_0]$ of $D(\mathbf{m},Q)$ is defined as $\frac{d}{k}K_{s}$, and $D[V_i]$ is defined as $m_iK_{k-s}$ for all $i \in [p]$.
\item For each $i \in [p]$, the arcs are constructed such that $V_i \xrightarrow{m_i} V_0$.
\item For each $v \in V_0$ and each $i \in [p]$, the arcs are constructed such that $\{v\} \xrightarrow{q_{vi}} V_i $.
%\item There are no other arcs in $D$.
\end{itemize}
\end{defi}

We explicitly characterize the multi-digraph $D$ in $\mathfrak{D}_d(\mathcal{S})$ by the multi-digraph $D(\mathbf{m},Q)$ constructed in Definition \ref{defi3.3}, and determine the multiplicity of each of its arcs from Lemma \ref{yinli3.2}.

\begin{pro}\label{pro3.4}
Let $\mathcal{S}=\mathcal{S}(k,s,p)$ be a sunflower, and let $d$ be a positive integer such that $k \mid d$. Then the set of multi-digraphs $\mathfrak{D}_d(\mathcal{S})$ is given by
$$\mathfrak{D}_d(\mathcal{S})=
\{D: D \cong D(\mathbf{m},Q), \mathbf{e}_p^\top\mathbf{m}=\frac{d}{k},Q\mathbf{e}_p=\frac{d}{k}\mathbf{e}_s,Q^{\top}\mathbf{e}_s=s\mathbf{m}\}.$$
\end{pro}

\begin{proof}
For any $D \in \mathfrak{D}_d(\mathcal{S})$, from Lemma \ref{yinli3.2}, it follows that there exist
$\mathbf{m}$ and $Q$ such that $D \cong D(\mathbf{m},Q)$.
Our first goal is to show that $\mathbf{e}_p^\top\mathbf{m}=\frac{d}{k},Q\mathbf{e}_p=\frac{d}{k}\mathbf{e}_s,Q^{\top}\mathbf{e}_s=s\mathbf{m}$.

For each $i \in [p]$, by \eqref{eqpi} and $m(u)=m_{i}$ for all $u \in P_{i}$,
we have $\sum_{v\in S}q_{vi}=sm_{i}$, that is, $Q^{\top}\mathbf{e}_s=s\mathbf{m}$.
For each vertex $v \in S$, from Lemma \ref{yinli3.2}(b) and (d), we have
$\mathrm{deg}^+(v)=(s-1)\frac{d}{k}+\sum_{i=1}^{p}(k-s)q_{vi}$.
Recall that $\mathrm{deg}^+(v)=(k-1)m(v)=(k-1)\frac{d}{k}$.
Hence, we obtain $\sum_{i=1}^{p}q_{vi}=\frac{d}{k}$ for each $v \in S$, that is, $Q\mathbf{e}_p=\frac{d}{k}\mathbf{e}_s$.
From Lemma \ref{yinli3.2}(b), (c) and (d),
we can get that the total number (counting multiplicities) of arcs in $D$ as
\begin{align*}
|E(D)|=s\frac{d}{k}(k-1)+\sum_{i=1}^{p}(k-s)m_{i}(k-1).
\end{align*}
And since $|E(D)|=d(k-1)$,
we obtain $\sum_{i=1}^{p}m_{i}=\frac{d}{k}$,
that is, $\mathbf{e}_p^\top\mathbf{m}=\frac{d}{k}$.
It implies that
$$\mathfrak{D}_d(\mathcal{S}) \subseteq
    \{D: D \cong D(\mathbf{m},Q), \mathbf{e}_p^\top\mathbf{m}=\frac{d}{k},Q\mathbf{e}_p=\frac{d}{k}\mathbf{e}_s,Q^{\top}\mathbf{e}_s=s\mathbf{m}\}.$$

Let $D(\mathbf{m},Q)$ satisfy $\mathbf{e}_p^\top\mathbf{m}=\frac{d}{k},Q\mathbf{e}_p=\frac{d}{k}\mathbf{e}_s$ and $Q^{\top}\mathbf{e}_s=s\mathbf{m}$.
To show that $D(\mathbf{m},Q) \in \mathfrak{D}_d(\mathcal{S}) $,
we construct $f \in \mathcal{F}_{d}(\mathcal{S})$ such that $D_{f} = D(\mathbf{m},Q)$.
For each $i \in [p]$,
let $e_i =S \cup P_i \in E(\mathcal{S})$.
Construct $f \in \mathcal{F}_{d}(\mathcal{S})$ as follows, for each $i \in [p]$ and each $u \in P_i$, assign exactly $m_i$ hyperedges $e_i$ rooted at $u$ in $f$; and for each $v \in S$ and each $i \in [p]$, assign exactly $q_{vi}$ hyperedges $e_i$ rooted at $v$.
It follows that $D_{f} = D(\mathbf{m},Q)$, which implies that
$$\mathfrak{D}_d(\mathcal{S}) \supseteq
\{D: D \cong D(\mathbf{m},Q), \mathbf{e}_p^\top\mathbf{m}=\frac{d}{k},Q\mathbf{e}_p=\frac{d}{k}\mathbf{e}_s,Q^{\top}\mathbf{e}_s=s\mathbf{m}\}.$$
\end{proof}

Using Proposition \ref{pro3.4}, we can give an intuitive description of the set $\mathcal{H}_d$.
Moreover, some parameters of the spectral moment coefficient are also provided,
including the number of spanning trees and the product of out-degrees of the vertices.

\begin{lem}\label{yinli3.5}
Let $\mathcal{S}=\mathcal{S}(k,s,p)$ be a sunflower, and let $d$ be a positive integer such that $k \mid d$. Then
\begin{enumerate}
\renewcommand{\labelenumi}{(\alph{enumi})}

\item The set $\mathcal{H}_d=\{\mathcal{S}(k,s,t): t\leq \min\{\frac{d}{k},p\}\}$.

\item If $p \leq \frac{d}{k}$, for $ D \cong D(\mathbf{m},Q) \in \mathfrak{D}_d(\mathcal{S})$, the number of spanning trees of $D$ is $$t(D)=d^{s-1}s^{p-1}k^{p(k-s-1)}\prod_{i=1}^{p}m_{i}^{k-s},$$
and
$$\prod_{v\in V(D)}\mathrm{deg}^{+}(v)=(\frac{d}{k})^s(k-1)^{kp-sp+s}\prod_{i=1}^{p}m_{i}^{k-s}.$$
    % Moreover,
    % $$|\{f:f\in\mathcal{F}_{d}(\mathcal{S}) \ \text{and} \ D_{f}\cong D\}|
    % =((k-1)!)^{d}(\frac{d}{k}!)^{s}\left(\prod_{v=1}^{s}\prod_{i=1}^{p}q_{vi}!\right)^{-1}.$$
% \item Let $D=D(\mathbf{m},Q)$. For $ D\in \mathfrak{D}_d(\mathcal{S})$, the number of spanning trees of $D$ is $t(D)=d^{s-1}s^{p-1}k^{p(k-s-1)}\prod_{t=1}^{p}m_{i}^{k-s}$.

% \item If $p \leq \frac{d}{k}$, for $ D=D(\mathbf{m},Q) \in \mathfrak{D}_d(\mathcal{S})$, then $$|\{f:f\in\mathcal{F}_{d}(\mathcal{S}) \ \text{and} \ D_{f}\cong D\}|=((k-1)!)^{d}(\frac{d}{k}!)^{s}\left(\prod_{v=1}^{s}\prod_{i=1}^{p}q_{vi}!\right)^{-1}.$$.

\end{enumerate}
\end{lem}

\begin{proof}
Note that $\mathcal{H}_{d}$ is a subset of connected subhypergraphs of $\mathcal{S}(k,s,p)$,
then we have $\mathcal{H}_{d}\subseteq\{\mathcal{S}(k,s,t): t\leq p\}$.
For $\mathcal{S}(k,s,t) \in \mathcal{H}_{d}$,
it follows from Proposition \ref{pro3.4} that there exists $t$-dimensional positive integer vector  $\mathbf{m}$ such that $\mathbf{e}^\top_t \mathbf{m}=\frac{d}{k}$.
Thus, we get $t \leq \frac{d}{k}$, and then $\mathcal{H}_d\subseteq\{\mathcal{S}(k,s,t): t\leq \min\{\frac{d}{k},p\}\}$.
On the other hand,
Proposition \ref{pro3.4} shows that $\mathfrak{D}_d(\mathcal{S}(k,s,t) )$ is not empty for any $t\leq \min\{\frac{d}{k},p\}$.
It implies $\mathcal{H}_d=\{\mathcal{S}(k,s,t): t\leq \min\{\frac{d}{k},p\}\}$.

Next, we will use the Matrix-Tree theorem \cite{duval2009simplicial} and an expression for the determinant involving the Schur complement \cite{brualdi1983determinantal} to give the number of spanning trees $t(D)$ for $ D \cong D(\mathbf{m},Q) \in \mathfrak{D}_d(\mathcal{S})$, when $p \leq \frac{d}{k}$.

Let $I_{t}$ and $J_{t}$ be an identity matrix and an all-ones matrix of size $t \times t$, respectively.
We write the Laplacian matrix of $D$ as a block matrix
\begin{align*}
L=\begin{bmatrix} M & N \\ X & Y \end{bmatrix}.
\end{align*}
where the matrix $M=L[S]$ is an $s \times s$ matrix given by $M=dI_{s}-\frac{d}{k}J_{s}$, i.e., the diagonal entries are $\frac{d}{k}(k-1)$ and the off-diagonal entries $-\frac{d}{k}$.
The matrix $Y=L[\cup_{i=1}^pP_i]$ is a diagonal block matrix with diagonal blocks $Y_{i}=L[P_i]=m_{i}(kI_{k-s}-J_{k-s})$ for all $i \in [p]$.

For each $i \in [p]$, let $\mathbf{q}_i$ denote the $i$-th column of $Q$,  so that $Q=[\mathbf{q}_1\cdots\mathbf{q}_p]$.
The matrix $N=\begin{bmatrix} N_{1}  \cdots  N_{p} \end{bmatrix}$ is a block matrix with blocks $N_{i}=-\mathbf{q}_i\mathbf{e}_{k-s}^{\top}$ for all $i \in [p]$.
The matrix $X=\begin{bmatrix} X_{1} \\ \vdots \\ X_{p} \end{bmatrix}$ is also a block matrix with blocks $X_{i}=-m_{i}\mathbf{e}_{k-s}\mathbf{e}_{s}^{\top}$ for all $i \in [p]$.

Let $\widehat{L}$ denote the submatrix of $L$ obtained by deleting the first row and column.
Then we can write it as block matrix
\begin{align*}
\widehat{L}=\begin{bmatrix} \widehat{M} & \widehat{N} \\ \widehat{X} & Y \end{bmatrix}.
\end{align*}
It is known that $t(D)=\mathrm{det}(\widehat{L})$ from the Matrix-tree Theorem \cite{duval2009simplicial}.
Note that all diagonal blocks $Y_{i}'s$ of $Y$ are invertible,
indeed, $Y_{i}^{-1}=(ks)^{-1}m_{i}^{-1}(sI_{k-s}+J_{k-s})$.
Then $Y$ is invertible.
From the determinant formula involving the Schur complement \cite{brualdi1983determinantal},
we have that
%\begin{align}\label{eq3}
%t(D)=\mathrm{det}(\widehat{L})
%=\mathrm{det}(Y)\mathrm{det}(\widehat{M}-\widehat{N}Y^{-1}\widehat{X}),
%\end{align}
\begin{equation}\label{eq3}
t(D)=\mathrm{det}(\widehat{L})=\begin{cases}
\mathrm{det}(Y)\mathrm{det}(\widehat{M}-\widehat{N}Y^{-1}\widehat{X}), &\text{if $s>1$},\\
\mathrm{det}(Y), &\text{if $s=1$},
\end{cases}
\end{equation}
where $\widehat{M}-\widehat{N}Y^{-1}\widehat{X}$ is the Schur complement of $Y$ in $\widehat{L}$.
Let $\hat{\mathbf{q}}_{i}$ be a vector obtained by truncating the first entry of $\mathbf{q}_{i}$.
Note that the matrix $\widehat{M}=dI_{s-1}-\frac{d}{k}J_{s-1}$,
the blocks of $\widehat{N}$ and $\widehat{X}$ are $\widehat{N_{i}}=-\hat{\mathbf{q}}_{i}\mathbf{e}_{k-s}^{\top}$ and $\widehat{X_{i}}=-m_{i}\mathbf{e}_{k-s}\mathbf{e}_{s-1}^{\top}$, respectively.
Then we have that
\begin{align*}
\widehat{N}Y^{-1}\widehat{X}&=\sum_{i=1}^{p}\widehat{N_{i}}Y_{i}^{-1}\widehat{X_{i}}\\
&=\sum_{i=1}^{p}\left(-\hat{\mathbf{q}}_{i}\mathbf{e}_{k-s}^{\top}\right)(ks)^{-1}m_{i}^{-1}(sI_{k-s}+J_{k-s})\left(-m_{i}\mathbf{e}_{k-s}\mathbf{e}_{s-1}^{\top}\right)\\
&=(ks)^{-1}\sum_{i=1}^{p}\hat{\mathbf{q}}_{i}\mathbf{e}_{k-s}^{\top}(sI_{k-s}+J_{k-s})\mathbf{e}_{k-s}\mathbf{e}_{s-1}^{\top}\\
&=(ks)^{-1}\sum_{i=1}^{p}\hat{\mathbf{q}}_{i}k(k-s)\mathbf{e}_{s-1}^{\top}.
\end{align*}

From Proposition \ref{pro3.4},
it follows that $Q\mathbf{e}_p=\frac{d}{k}\mathbf{e}_s$.
By removing the first entry of each column, we have $\sum_{i=1}^{p}\hat{\mathbf{q}}_{i}\mathbf{e}_{s-1}^{\top}=\frac{d}{k}J_{s-1}$,
which implies that $\widehat{N}Y^{-1}\widehat{X}=\frac{d(k-s)}{ks}J_{s-1}$.
Then we get $\widehat{M}-\widehat{N}Y^{-1}\widehat{X}=dI_{s-1}-\frac{d}{s}J_{s-1}$, and hence
\begin{align}\label{eq4}
\mathrm{det}(\widehat{M}-\widehat{N}Y^{-1}\widehat{X})=\mathrm{det}(dI_{s-1}-\frac{d}{s}J_{s-1})=\frac{d^{s-1}}{s}.
\end{align}
Note that $\mathrm{det}(Y_{i})=\mathrm{det}(m_{i}(kI_{k-s}-J_{k-s}))=m_{i}^{k-s}sk^{k-s-1}$ for all $i \in [p]$,
then we have
\begin{align}\label{eq5}
\mathrm{det}(Y)=\prod_{i=1}^{p}\mathrm{det}(Y_{i})=(sk^{k-s-1})^{p}\prod_{i=1}^{p}m_{i}^{k-s}.
\end{align}
Substituting (\ref{eq4}) and (\ref{eq5}) into (\ref{eq3}),
we obtain the expression for $t(D)$.

Finally, from  Proposition \ref{pro3.4},
we observe that the out-degree $\mathrm{deg}^{+}(v)=\frac{d}{k}(k-1)$ for all $v \in S$
and $\mathrm{deg}^{+}(u)=m_{i}(k-1)$ for all $u \in P_{i}$ and each $i \in [p]$.
Thus, we get
\begin{align*}
\prod_{v\in V(D)}\mathrm{deg}^{+}(v)=(\frac{d}{k}(k-1))^{s}\prod_{i=1}^{p}(m_{i}(k-1))^{k-s}
=(\frac{d}{k})^s(k-1)^{kp-sp+s}\prod_{i=1}^{p}m_{i}^{k-s}.
\end{align*}
\end{proof}

For positive integers $p$ and $s$,
recall that
\[\Xi^{p}=\{(\xi_{i})\in\mathbb{C}^{p}: \mbox{$\xi_{i}^{s+1}-\xi_{i}=0$ for all $i\in[p]$}\},\]
and $\mathbf{e}_{p}$ denotes the $p$-dimensional all-ones column vector.
We are now ready to give an expression for the spectral moments of the sunflower as follows.

\begin{thm}\label{dingli2}

% Let $\mathcal{S}=\mathcal{S}(k,s,p)$ be a $k$-uniform sunflower with $s$ seeds and $p$ petals.
% Denote $\Xi^{p}=\{(\xi_{i})\in\mathbb{C}^{p}:\mbox{$\xi_{i}^{s+1}-\xi_{i}=0$ for all $i\in[p]$} \}$
% and let $\mathbf{e}_{p}$ denote the $p$-dimensional all-ones vector.
The $d$-th order spectral moment of the sunflower $\mathcal{S}=\mathcal{S}(k,s,p)$ is
\begin{equation*}
\mathrm{S}_{d}(\mathcal{S})=\begin{cases}
\sum_{\xi\in\Xi^{p}}\frac{1}{s}K^{p-|\mathrm{supp}(\xi)|}k^{|\mathrm{supp}(\xi)|(k-s-1)+s}(\mathbf{e}_{p}^{\top}\xi)^{\frac{sd}{k}},
&\text{if $k\mid d$},\\
0,
&\text{if $k\nmid d$},
\end{cases}
\end{equation*}
where $K=(k-1)^{k-s}-sk^{k-s-1}$.

\end{thm}

\begin{proof}

Recall that $\mathrm{S}_{d}(\mathcal{S})=0$ for $k\nmid d$,
since $\mathcal{S}$ is a cored hypergraph.
Thus, in the remainder of the proof, we only consider the case $k\mid d$.

We write $\mathcal{S}_t=\mathcal{S}(k,s,t)$ for brevity.
From Lemma \ref{yinli1} and  Lemma \ref{yinli3.5}(a), we have
\begin{align}\label{eq6}
\mathrm{S}_{d}(\mathcal{S}_p)=(k-1)^{p(k-s)+s-1}\sum_{t\leq \min\{\frac{d}{k},p\}}\binom{p}{t}c_{d}(\mathcal{S}_t),
\end{align}
where
\begin{align}\label{eq7}
c_{d}(\mathcal{S}_t)=d(k-1)((k-1)!)^{-d}\sum_{D\in\mathfrak{D}_{d}(\mathcal{S}_t)}\frac{|\{f:\mbox{$f\in\mathcal{F}_{d}(\mathcal{S}_t)$ and $D_{f}\cong D$}\}|t(D)}{\prod_{v\in V(D)}\mathrm{deg}^{+}(v)}.
\end{align}
From Lemma \ref{yinli3.5}(b), we have that
\begin{align}\label{eq8}
\frac{t(D)}{\prod_{v\in V(D)}\mathrm{deg}^{+}(v)}=\frac{s^{t-1}k^{t(k-s-1)+s}}{d(k-1)^{t(k-s)+s}}.
\end{align}
Substituting \eqref{eq8} to \eqref{eq7}, we have that
\begin{align}\label{eq12}
  c_{d}(\mathcal{S}_t)=\frac{s^{t-1}k^{t(k-s-1)+s}}{((k-1)!)^{d}(k-1)^{t(k-s)+s-1}}\sum_{D\in\mathfrak{D}_{d}(\mathcal{S}_t)}|\{f:\mbox{$f\in\mathcal{F}_{d}(\mathcal{S}_t)$ and $D_{f}\cong D$}\}|.
\end{align}
Let $\mathbb{Z}_{+}$ and $\mathbb{N}$ denote  the sets of positive integers and natural numbers, respectively.
For any $D \in \mathfrak{D}_d(\mathcal{S}_t)$, Proposition \ref{pro3.4} implies that $D \cong D(\mathbf{m},Q)$
for some $\mathbf{m} \in \mathbb{Z}_{+}^t$ satisfying $\mathbf{e}_t^\top\mathbf{m}=\frac{d}{k}$ and $Q \in \mathbb{N}^{s \times t}$ satisfying $Q\mathbf{e}_t=\frac{d}{k}\mathbf{e}_s,Q^{\top}\mathbf{e}_s=s\mathbf{m}$.
Thus, we have that
\begin{align*}
  &\sum_{D\in\mathfrak{D}_{d}(\mathcal{S}_t)}|\{f:\mbox{$f\in\mathcal{F}_{d}(\mathcal{S}_t)$ and $D_{f}\cong D$}\}| \\ &=\sum_{\mathbf{m}:\mathbf{e}_t^\top\mathbf{m}=\frac{d}{k}}  |\{f\in\mathcal{F}_{d}(\mathcal{S}_t):\mbox{  $D_{f} = D(\mathbf{m},Q)$, $Q\mathbf{e}_t=\frac{d}{k}\mathbf{e}_s$ and $Q^{\top}\mathbf{e}_s=s\mathbf{m}$}\}| .
\end{align*}

Let $e_{i}=S \cup P_{i} \in E(\mathcal{S}_t)$ for each $i \in [t]$.
For $f=(v_1\alpha_1,v_2\alpha_2,\ldots,v_d\alpha_d) \in \mathcal{F}_{d}(\mathcal{S}_t)$, recall that $v_1\leq v_2\leq \cdots \leq v_d$ and
the $k$-tuple $v_j\alpha_j$ represent a hyperedge $e_i \in E(\mathcal{S}_t)$ rooted at $v_j$.
For each $i \in [t]$ and each $v \in P_i$, we observe that all hyperedges with root $v$ are $e_i$ in $f$. Moreover,
Lemma  \ref{yinli3.2}(c) implies that $m_i$ is the number of such hyperedges $e_i$ with root $v$  in $f$.
For each $v \in S$,  Lemma  \ref{yinli3.2}(b) implies that $q_{vi}$ is the number of hyperedges $e_i$ with root $v$.

In order to count the number of $f$ that give rise to a given $D(\mathbf{m},Q)$, note that we can permute the $k-1$ non-roots of each hyperedges (without changing  $D_f$).
We can also permute hyperedges with the same root.
If this root is not a seed, then this however leads to the same $f$ because the permuted hyperedges are the same.
If the root, $v$ say, is a seed, then it occurs $\sum_{i \in [t]}q_{vi}$ times, but again, permuting the same hyperedges (with hyperedge $e_i$ with root $v$ occurring $q_{vi}$ times) leads to same $f$. It implies that
\begin{align*}
   &|\{f\in\mathcal{F}_{d}(\mathcal{S}_t):\mbox{  $D_{f} = D(\mathbf{m},Q)$, $Q\mathbf{e}_t=\frac{d}{k}\mathbf{e}_s$ and $Q^{\top}\mathbf{e}_s=s\mathbf{m}$}\}| \\
   & = ((k-1)!)^{d}\frac{(\sum_{v\in S}\sum_{i\in[t]}q_{vi})!}{\prod_{i=1}^{t}(\sum_{v \in S}q_{vi})!}\\
   &=((k-1)!)^{d}\frac{(s\frac{d}{k})!}{\prod_{i=1}^{t}(sm_{i})!}.
\end{align*}
Here, we write $\phi(\mathbf{m})=\frac{(s\frac{d}{k})!}{\prod_{i=1}^{t}(sm_{i})!}$.
Let $\mathbb{M}^t_{+}=\{\mathbf{m} \in \mathbb{Z}_{+}^t: \mathbf{e}_t^\top\mathbf{m}=\frac{d}{k} \}$, and let $\mathbb{M}^t_{\geq 0}=\{\mathbf{m} \in \mathbb{N}^t: \mathbf{e}_t^\top\mathbf{m}=\frac{d}{k} \}$.
Then we have that
\begin{align*}
\sum_{D\in\mathfrak{D}_{d}(\mathcal{S}_t)}|\{f:\mbox{$f\in\mathcal{F}_{d}(\mathcal{S}_t)$ and $D_{f}\cong D$}\}| =((k-1)!)^{d}\sum_{\mathbf{m} \in \mathbb{M}^t_{+}}\phi(\mathbf{m}).
\end{align*}
Note that $\sum_{\mathbf{m} \in \mathbb{M}^t_{\geq 0}}\phi(\mathbf{m})=\sum_{i=1}^{t}\binom{t}{i}\sum_{\mathbf{m} \in \mathbb{M}^i_{+}}\phi(\mathbf{m})$.
From the Principle of Inclusion-Exclusion \cite{Stanley2012enumerative}, we have that
\begin{align*}
\sum_{\mathbf{m} \in \mathbb{M}^t_{+}}\phi(\mathbf{m}) =\sum_{i=1}^{t}(-1)^{t-i}\binom{t}{i}\sum_{\mathbf{m} \in \mathbb{M}^i_{\geq 0}}\phi(\mathbf{m}).
\end{align*}
Recall that $\Xi_{0}^{t}=\{\xi\in\Xi^{t}: \mbox{$|\mathrm{supp}(\xi)|=0$ or $t$ }\}$.
Note that $\phi(\mathbf{m})$ is the so-called \emph{multinomial coefficient} \cite{Stanley2012enumerative}, it implies that
$$\sum_{\mathbf{m} \in \mathbb{M}^t_{\geq 0}}\phi(\mathbf{m})
=\frac{1}{s^{t}}\sum_{\xi\in\Xi^{t}_0 \setminus \{\mathbf{0}\}}(\mathbf{e}_{t}^{\top}\xi)^{s\frac{d}{k}}.$$
By \eqref{eq12}, we have
\begin{align*}
c_{d}(\mathcal{S}_t)
=\frac{s^{t-1}k^{t(k-s-1)+s}}{(k-1)^{t(k-s)+s-1}}\sum_{i=1}^{t}(-1)^{t-i}\binom{t}{i}\frac{1}{s^{i}}
\sum_{\xi\in\Xi^{i}_{0} \setminus \{\mathbf{0}\}}(\mathbf{e}_{i}^{\top}\xi)^{s\frac{d}{k}}.
\end{align*}
Substituting this into \eqref{eq6}, we obtain
\begin{align*}
&\mathrm{S}_{d}(\mathcal{S}) \\
&=\frac{k^{s}}{s}\sum_{t=1}^{p}\binom{p}{t}(k-1)^{(k-s)(p-t)}(sk^{k-s-1})^{t}\sum_{i=1}^{t}(-1)^{t-i}\binom{t}{i}\frac{1}{s^{i}}\sum_{\xi\in\Xi^{i}_{0} \setminus \{\mathbf{0}\}}(\mathbf{e}_{i}^{\top}\xi)^{s\frac{d}{k}}\\
&=\frac{k^{s}}{s}\sum_{i=1}^{p}\sum_{t=i}^{p}\binom{p}{t}\binom{t}{i}(k-1)^{(k-s)(p-t)}(sk^{k-s-1})^{t}(-1)^{t-i}\frac{1}{s^{i}}\sum_{\xi\in\Xi^{i}_{0} \setminus \{\mathbf{0}\}}(\mathbf{e}_{i}^{\top}\xi)^{s\frac{d}{k}}.
\end{align*}
Since $\binom{p}{t}\binom{t}{i}=\binom{p}{i}\binom{p-i}{t-i}$ for $i \leq t\leq p$,
we have that
\begin{align*}
\mathrm{S}_{d}(\mathcal{S})&=\sum_{i=1}^{p}\binom{p}{i}\frac{k^{s}}{s}K^{p-i}(k^{k-s-1})^{i}\sum_{\xi\in\Xi^{i}_{0} \setminus \{\mathbf{0}\}}(\mathbf{e}_{i}^{\top}\xi)^{s\frac{d}{k}}\\
&=\sum_{\xi\in\Xi}\frac{k^{s}}{s}K^{p-|\mathrm{supp}(\xi)|}k^{|\mathrm{supp}(\xi)|(k-s-1)}(\mathbf{e}_{p}^{\top}\xi)^{\frac{sd}{k}},
%&=\frac{k^{s}}{s}\sum_{i=1}^{p}\binom{p}{i}(-1)^{i}\frac{1}{s^{i}}\sum_{\xi\in\Xi^{i}_{0} \setminus \{\mathbf{0}\}}(\mathbf{e}_{i}^{\top}\xi)^{s\frac{d}{k}}\sum_{t=i}^{p}(-1)^{t}\binom{p-i}{t-i}(k-1)^{(k-s)(p-t)}(sk^{k-s-1})^{t}.
\end{align*}
where we use the identity $$\sum_{t=i}^{p}(-1)^{t-i}\binom{p-i}{t-i}(k-1)^{(k-s)(p-t)}(sk^{k-s-1})^{t-i}=((k-1)^{k-s}-sk^{k-s-1})^{p-i}.$$
\end{proof}

By Theorems \ref{dingli1} and \ref{dingli2},
we obtain the characteristic polynomial of the sunflower as follows,
which is the main result in this paper.

\begin{thm}\label{mainresult}

The characteristic polynomial of the sunflower $\mathcal{S}=\mathcal{S}(k,s,p)$ is

\begin{align*}
\phi_{\mathcal{S}}(\lambda)=\prod_{\xi\in\Xi^{p}}(\lambda^{k}-(\mathbf{e}_{p}^{\top}\xi)^{s})^{\mu(\xi)},
\end{align*}
where
\begin{equation*}
\mu(\xi)=\begin{cases}
\frac{(p(k-s)+s)}{k}(k-1)^{p(k-s)+s-1}-\frac{k^{s-1}}{s}((k-1)^{p(k-s)}-K^{p}),
&\text{if $\xi=\mathbf{0}$},\\
\frac{1}{s}K^{p-|\mathrm{supp}(\xi)|}k^{|\mathrm{supp}(\xi)|(k-s-1)+s-1},
&\text{if $\xi\neq\mathbf{0}$},
\end{cases}
\end{equation*}
and $K=(k-1)^{k-s}-sk^{k-s-1}$.

\end{thm}

\begin{proof}
When $k=2$, we have $s=1$.
In this case, the sunflower $\mathcal{S}$ reduces to a star graph.
It is known that $\lambda^{p-1}(\lambda^{2}-p)$ is the characteristic polynomial of $\mathcal{S}$ \cite{cvetkovic1980spectra}.
In the following proof, we only consider the case where $k \geq 3$.

From the complete set of eigenvalues of $\mathcal{S}$ (see Theorem  \ref{dingli1}) and the $k$-symmetry of its spectrum,
it follows that the characteristic polynomial of $\mathcal{S}$ of the form
\begin{align*}
\phi(\lambda)=\prod_{\xi\in\Xi^{p}}(\lambda^{k}-(\mathbf{e}_{p}^{\top}\xi)^{s})^{\mu(\xi)}.
\end{align*}
Note that $\lambda^k-(\mathbf{e}_{p}^{\top}\xi)^{s}$ does not necessarily give rise to the factor of $\phi(\lambda)$,
if $s=k-1$ and $\xi \notin \Xi_0^p$. In this case, we have $\mu(\xi)=0$.

In what follows, we verify that the degree of $\phi(\lambda)$ is $(p(k-s)+s)(k-1)^{p(k-s)+s-1}$,
and that the sum of $d$-th powers of its roots coincides with $\mathrm{S}_{d}(\mathcal{S})$. This will complete the proof.

By Theorem \ref{dingli2}, it can be readily verified that the sum of $d$-th powers of the roots of $\phi(\lambda)$ coincides with the spectral moment $\mathrm{S}_d(\mathcal{S})$.
The total multiplicity of the factor $\lambda^{k}-(\mathbf{e}_{p}^{\top}\xi)^{s}$, as $\xi$ ranges over $\Xi^{p} \setminus \{\mathbf{0}\}$, is given by
% \begin{align*}
% \sum_{\xi \in \Xi^{p} \setminus \{\mathbf{0}\}}\mu(\xi)=&\sum_{\xi \in \Xi^{p} \setminus \{\mathbf{0}\}}\binom{p}{|\mathrm{supp}(\xi)|}s^{|\mathrm{supp}(\xi)|}\frac{1}{s}K^{p-|\mathrm{supp}(\xi)|}k^{|\mathrm{supp}(\xi)|(k-s-1)+s-1} \\
% &=\frac{k^{s-1}}{s}\sum_{i=1}^{p}\binom{p}{i}\left((k-1)^{k-s}-sk^{k-s-1}\right)^{p-i}(sk^{k-s-1})^{i}\\
% &=\frac{k^{s-1}}{s}((k-1)^{p(k-s)}-K^{p}).
% \end{align*}
\begin{align*}
\sum_{\xi \in \Xi^{p} \setminus \{\mathbf{0}\}}\mu(\xi)
&=\sum_{\xi \in \Xi^{p} \setminus \{\mathbf{0}\}}\frac{1}{s}K^{p-|\mathrm{supp}(\xi)|}k^{|\mathrm{supp}(\xi)|(k-s-1)+s-1} \\
&=\sum_{i=1}^{p}\binom{p}{i}s^{i}\frac{1}{s}K^{p-i}k^{i(k-s-1)+s-1}\\
&=\frac{k^{s-1}}{s}\sum_{i=1}^{p}\binom{p}{i}\left((k-1)^{k-s}-sk^{k-s-1}\right)^{p-i}(sk^{k-s-1})^{i}\\
&=\frac{k^{s-1}}{s}((k-1)^{p(k-s)}-K^{p}).
\end{align*}
We conclude that the degree of $\phi(\lambda)$ is
$$\deg\phi(\lambda)=k(\mu(\mathbf{0})+\sum_{\xi \in \Xi^{p} \setminus \{\mathbf{0}\}}\mu(\xi))=(p(k - s) + s)(k - 1)^{p(k - s) + s - 1}.$$
\end{proof}

It is known that the spectral radius of a connected graph is an eigenvalue of the graph with algebraic multiplicity $1$.
From the Perron-Frobenius Theorem for tensors,
the spectral radius of a connected uniform hypergraph is also one of its eigenvalues \cite{chang2008perron}.
However, its algebraic multiplicity remains unknown in general cases.
Currently, the algebraic multiplicities of the spectral radius have been determined for some classes of hypergraphs,
including $k$-uniform power hypergraphs \cite{chen2024spectra}, $k$-uniform hypertrees \cite{chen2024multiplicity} and $3$-uniform complete hypergraphs \cite{zheng2021complete}.
We derive the following result from Theorem \ref{mainresult},
which confirms Fan's conjecture \citep[Conjecture 5.5]{fan2024themultiplicity} in the case of sunflowers.

\begin{cor}
The spectral radius of the sunflower $\mathcal{S}(k,s,p)$ is $\sqrt[k]{p^s}$ and its algebraic multiplicity is $k^{p(k-s)+s-1-p}$.
\end{cor}

\section*{Acknowledgment}

This work is supported by the National Natural Science Foundation of China (No. 12071097, 12371344),
the Natural Science Foundation for The Excellent Youth Scholars of the Heilongjiang Province (No. YQ2022A002, YQ2024A009),
the China Postdoctoral Science Foundation (No. 2024M761510) and the Fundamental Research Funds for the Central Universities.

\bibliographystyle{plain}
\bibliography{sunbib}

\end{document}